\documentclass{amsart}
\usepackage[utf8]{inputenc}
\usepackage{amssymb}
\usepackage{lineno,hyperref}
\usepackage{xcolor}
\usepackage{graphicx}
\usepackage{caption}
\usepackage{subcaption}
\usepackage{amsmath}
\usepackage{amssymb,amsthm}
\usepackage{amsmath}
\usepackage{cite}
\usepackage{multicol}

\newtheorem{theorem}{Theorem}

\newtheorem{lemma}{Lemma}

\newtheorem{remark}{Remark}

\title[Existence of periodic solution of Nicholson--type system]{Existence of periodic solution of Nicholson--type system with nonlinear density-dependent mortality}
\author{Gustavo Ossand\'on}
\author{Daniel Sep\'ulveda}

\email{gusosar@utem.cl, daniel.sepulveda@utem.cl}

\address{Departamento de Matem\'atica, Universidad Tecnol\'ogica Metropolitana, Las Palmeras 3360, \~Nu\~noa, Santiago, Chile.}


\begin{document}

\maketitle
\begin{abstract}
This article studies an $\omega$--periodic system of Nicholson--type 
differential equations with nonlinear density-dependent mortality 
rate. Using the degree theory we obtain sufficient 
conditions for the existence of a positive solution $\omega$--periodic. Our result complement previous researches on the subject.   
\end{abstract}

\section{Introduction}

In \cite{berezansky2011} Berezansky, Idels and  Troib introduce
the system of differential equations with delay given by
\begin{equation}\label{Nich Sys}
\begin{array}{ccc}
x_1'(t)&=&-d_1x_1(t)+b_1x_1(t-\tau)\exp(-x_1(t-\tau))+m_{12}x_2(t)\\  \\
x_2'(t)&=&-d_2x_2(t)+b_2x_2(t-\tau)\exp(-x_2(t-\tau))+m_{21}x_1(t).
\end{array}
\end{equation}
The system \eqref{Nich Sys} is called  Nicholson--type, since the delayed 
recruitment rate has the structure considered in Nicholson's equation introduced by Gurney, Blythe and Nisbet in \cite{gurney1980}. System \eqref{Nich Sys} corresponds to a population model with two patches, and is inspired by models
of marine protected areas and also by compartmental models for the growth of cancer cells.

In 2010 Berezansky, Braverman and Idels proposed a series of open problems related
to Nicholson's equation,  among them they suggest consider a Nicholson model with a non-linear mortality term depending on density, see \cite[pp. 1416]{berezansky2010}.
In 2011 Liu and Gong \cite{liu2011} introduced a Nicholson--type system with density-dependent nonlinear mortality rates:
\begin{equation}\label{Nich Sys DDM}
\begin{array}{ccc}
x_1'(t)&=&-m_{11}(t,x_1(t))+b_1(t)x_1(t-\tau_1(t))\exp(-x_1(t-\tau_1(t)))+m_{12}(t,x_2(t))\\  \\
x_2'(t)&=&-m_{22}(t,x_2(t))+b_2(t)x_2(t-\tau_2(t))\exp(-x_2(t-\tau_2(t)))+m_{21}(t,x_1(t)),
\end{array}
\end{equation}
with 
$$m_{ij}(t,x)=\frac{\delta_{ij}(t)x}{c_{ij}(t)+x}\mbox{ or }m_{ij}(t,x)=\delta_{ij}(t)-c_{ij}\exp(-x),$$
where $\delta_{ij},c_{ij}, b_i:\mathbb{R}\to (0,+\infty)$  are all positive, bounded and continuous functions,  and the functions  $\tau_i:\mathbb{R}\to [0,\infty)$ are continuous and bounded, $r_i=\sup_{t\in \mathbb{R}}\tau_i(t)>0,$ and $1\leq i,j\leq 2$, see \cite{liu2011} for more details.

There are several works about Nicholson--type systems with nonlinear mortality among others \cite{chen2012permanence} where permanence is studied,  \cite{liu2017global} consider the stability and almost periodicity, and  \cite{yao2018} addressed  the global attractivity.   

In this work we will prove the existence of $\omega$--periodic solutions for the Nicholson--type system  given by,
\begin{equation}\label{eq: main}
\begin{array}{ccc}
x_1'(t)&=&-\dfrac{\delta_{11}(t)x_1(t)}{c_{11}(t)+x_1(t)} +b_1(t)f(x_1(t-\tau_1(t)) )+\dfrac{\delta_{12}(t)x_2(t)}{c_{12}(t)+x_2(t)}\\  \\
x_2'(t)&=&-\dfrac{\delta_{22}(t)x_2(t)}{c_{22}(t)+x_2(t)} +b_2(t)f(x_2(t-\tau_2(t)) )+\dfrac{\delta_{21}(t)x_1(t)}{c_{21}(t)+x_1(t)},
\end{array}
\end{equation}
where $f(x)=x\exp(-x)$ and $\delta_{ij},c_{ij}, b_i,\tau_i:\mathbb{R}\to (0,+\infty)$  are all positive, bounded continuous and $\omega$--periodic functions for $1\leq i,j\leq 2$.

The system \eqref{Nich Sys} with periodic coefficients has been studied, among others, by Faria \cite{faria2017}  and Amster and Deboli \cite{amster2016}.
In the first work the author, by using a dissipativity result, the Green operator of the system and the Schauder fixed point Theorem  manages to determine sufficient conditions to guarantee the existence of periodic solutions. 
In the second article the authors study a planar system with nonlinear harvesting terms, they use topological degree theory to determine sufficient conditions for the existence of periodic solutions, the authors also determine necessary conditions for the existence of periodic solutions. 

Systems similar to \eqref{eq: main} with periodic coefficients have been considered in the work of Chen and Wang \cite{chen2012positive}, they obtained sufficient conditions for the existence of periodic solutions through degree theory. Recently, Son et al. \cite{son2019} consider a scalar version of \eqref{Nich Sys DDM} and study the global attractivity and peridiocity of the solutions. We point out that our result complements
those obtained in \cite{chen2012positive}.

Let us introduce some notation. As usual, the closure  and the boundary of a subset $A$ of topological space shall be denoted by $\overline{A}$ and $\partial A $, respectively. Let 
$$C_{\omega}:=\{X=(x_1,x_2)\in C(\mathbb{R},\mathbb{R}^2): X(t+\omega)=X(t)\mbox{ for all }t\in\mathbb{R}\}$$ be the Banach space of $\omega$--periodic continuos vector functions with the norm 
$$||X||=\max\{\sup_{t\in [0,\omega]}||x_1(t)||,\sup_{t\in [0,\omega]}||x_2(t)||\}.$$
Hereafter, for any continuous and $\omega$--periodic function   $x\in C(\mathbb{R},\mathbb{R})$ we denote
 $$x^+:=\max_{t\in[0,\omega]}\{x(t)\},\quad x^-:=\min_{t\in[0,\omega]}\{x(t)\}.$$

The organization of the remainder of the paper is as  follows. In section 2 we have compiled some basic facts  of degree theory and summarize two results on which the proof of our Theorem is supported.       In the third section our existence result is stated and proved. Finally, the fourth section contains a brief discussion about our result and other related ones that can be found in the references, we conclude with numerical simulations to illustrate our Theorem.

\section{Preliminaries}

The existence of $\omega$--periodic solutions for the system \eqref{eq: main} will be proved by using a \textit{continuation Theorem}, a kind of result that belong to the topological degree theory. 
In this section we  present main aspects of  degree theory and some fundamental results that will be used in our work.
 \subsection*{Brouwer degree} Let $\Omega\subset \mathbb{R}^n$ be open and bounded, let $f:\overline{\Omega}\to\mathbb{R}^n$ be a $C^1$ mapping, and let $\mathbf{y}\notin f(\partial \Omega)$ be a regular value of $f$. As usual, the Brouwer degree of $f$ at $\mathbf{y}$ over $\Omega$ is defined by
$$\deg_B(f,\Omega,\mathbf{y}):=\sum_{\mathbf{x}\in f^{-1}(\mathbf{y})} \mathrm{sgn} (Jf(\mathbf{x})),  $$
where $Jf(\mathbf{x}):=\det (Df(\mathbf{x}))$ denotes the Jacobian determinat of $f$ at $\mathbf{x}$. We note that $f$ has at least one zero in $\Omega$ if and only if $\deg_B(f,\Omega,0)\neq 0$.\\
\subsection*{Leray-Schauder degree}
 Let $\mathbb{X}$ a Banach space and $\Omega\subset \mathbb{X}$ open and bounded, and let $K:\partial \Omega\to \mathbb{X}$ be a compact operator such that $K\mathbf{x}\neq \mathbf{x}$ for $x\in \partial\Omega$. Set $\varepsilon=\inf_{\mathbf{x}\in \partial \Omega}||\mathbf{x}-K\mathbf{x} ||$, and define 
 $$\deg_{LS} (I-K,\Omega,0)=\deg_B((I-K_{\varepsilon})|_{V_{\varepsilon}},\Omega\cap V_{\varepsilon},0 ),$$
 where $K_{\varepsilon}$ is an $\varepsilon$--approximation of $K$ with $\mathrm{Im}(K_{\varepsilon})\subset V_{\varepsilon}$ and $\mathrm{dim}(V_{\varepsilon})<\infty$.\\
Since the definition of Leray--Schauder degree it follows that, if $K(\overline{\Omega})\subset V$, with $V\subset\mathbb{X}$ a finite dimensional subspace, then
$$\deg_{LS}(I-K,\Omega,0)=\deg_B((I-K)|_{\overline{\Omega}\cap V},\Omega\cap V,0).$$

\noindent

\subsection*{Fundamental Results} In order to state the continuation Theorem for our model we introduce a family of differential equations and an  operator 
$\Phi:=(\phi_1,\phi_2):C_{\omega}\to C_{\omega}$ given respectively by  
\begin{equation}\label{eq: main lam}
\begin{array}{ccc}
x_1'(t)&=& \lambda \left(-\dfrac{\delta_{11}(t)x_1(t)}{c_{11}(t)+x_1(t)} +b_1(t)f(x_1(t-\tau_1(t)) )+\dfrac{\delta_{12}(t)x_2(t)}{c_{12}(t)+x_2(t)}\right)\\  \\
x_2'(t)&=&\lambda \left(-\dfrac{\delta_{22}(t)x_2(t)}{c_{22}(t)+x_2(t)} +b_2(t)f(x_2(t-\tau_2(t)) )+\dfrac{\delta_{21}(t)x_1(t)}{c_{21}(t)+x_1(t)}\right) ,
\end{array}
\end{equation}
with $\lambda \in (0,1)$,  and
\begin{eqnarray*}
\phi_1(x_1,x_2)(t)&:=&-\dfrac{\delta_{11}(t)x_1(t)}{c_{11}(t)+x_1(t)} +b_1(t)f(x_1(t-\tau_1(t)) )+\dfrac{\delta_{12}(t)x_2(t)}{c_{12}(t)+x_2(t)},\\
\phi_2(x_1,x_2)(t)&:=&-\dfrac{\delta_{22}(t)x_2(t)}{c_{22}(t)+x_2(t)} +b_2(t)f(x_2(t-\tau_2(t)) )+\dfrac{\delta_{21}(t)x_1(t)}{c_{21}(t)+x_1(t)}.
\end{eqnarray*}
Then, following Amster \cite[pp. 150]{amster2014}, the corresponding continuation Theorem for system \eqref{eq: main} is
\begin{lemma}[Continuation Theorem]\label{thm EDR resonante }
Assume there exists an open bounded $\Omega\subset C_{\omega}$ such that:
\begin{enumerate}
    \item[a)] The system \eqref{eq: main lam} has no solutions on  $\partial \Omega$  for $\lambda \in (0,1).$
    \item[b)] $g(\mathbf{x})\neq 0$ for $\mathbf{x}\in \partial \Omega\cap \mathbb{R}^2$, where $g: \mathbb{R}^2\mapsto\mathbb{R}^2$ is given by:
    $$g(\mathbf{x}):=-\frac{1}{{\omega}}\int_0^{\omega}\Phi(\mathbf{x})dt;$$
    \item[c)] $deg_B(g,\Omega\cap \mathbb{R}^2,0)\neq 0.$
\end{enumerate}
Then \eqref{eq: main} has at least one solution $X\in \overline{\Omega}.$
\end{lemma} 
Next we state the \emph{Poincar\'e--Miranda Theorem} in $\mathbb{R}^2$, this result  is an extension of the intermediate value Theorem for $\mathbb{R}^n$. For a fuller treatment of the Poincar\'e--Miranda Theorem we refer the reader to \cite{kulpa1997}.
\begin{lemma}[Poincar\'e--Miranda Theorem]\label{Poinc Mir}
Let $f: [a,b]\times [a,b]\to \mathbb{R}^2$, $f=(f_1,f_2)$, be a continuous map such that
\begin{equation}\label{1.7}
f_1(a,x_2)<0<f_1(b,x_2)\quad \forall \, x_2\in [a,b]    
\end{equation}
and 
\begin{equation}
f_2(x_1,a)<0<f_2(x_1,b)\quad \forall \, x_1\in [a,b].    
\end{equation}
Then there exist $\mathbf{x}\in (a,b)\times (a,b)$ such that $f(\mathbf{x})=(0,0).$
\end{lemma}

\section{Existence of a positive periodic solution}

In order to prove the existence of periodic solution in nonlinear differential equations, fixed point Theorems are often used, see for instance \cite[Chap. 3]{burton2014}.
We rely on the theory of topological degree, which is closely connected with fixed point theory.
To achieve our main result we need the following assumptions:
   \begin{multicols}{2}
 \begin{itemize}
      \item[(H1)]\quad $\dfrac{b_1^+}{e} +\delta_{12}^+< \delta_{11}^-$,\\
     \item[(H3)] \quad $\dfrac{\delta_{11}^+}{c_{11}^-}-\dfrac{\delta_{12}^-}{c_{12}^+} < b_1^-$, \\
     \item[(H2)] \quad $\dfrac{b_2^+}{e} +\delta_{21}^+< \delta_{22}^-$,\\
      \item[(H4)]\quad $\dfrac{\delta_{22}^+}{c_{22}^-}-\dfrac{\delta_{21}^-}{c_{21}^+} < b_2^-$.
          \end{itemize}
        \end{multicols}

We state our result of existence in the following theorem:
\begin{theorem}\label{main thm}
Assume that \emph{(H1)--(H4)} are satisfied. 
Then the system \eqref{eq: main} has an $\omega$--periodic solution. 
\end{theorem}
\begin{proof}
The proof will be divided into three parts, one for each condition of the continuation Theorem. We first consider the open and bounded set $\Omega:=\Omega(\varepsilon,R)$ defined by
\begin{equation}\label{def Omega}
  \Omega:=\{(x_1(t),x_2(t))\in C_{\omega}\,:\, \varepsilon< x_i(t)<R, \mbox{ for }t\in[0,\omega], \, i=1,2 \} 
\end{equation}

\noindent\textbf{\textit{a) We will show that the system \eqref{eq: main lam} has no solution in $\partial \Omega$ with $\lambda \in (0,1)$.}}  For this we will obtain the a priori upper and lower bounds.\\
\noindent
\\
\textit{A priori upper  bounds.} We will prove that if $R$ is large enough then the $\omega$--periodic solutions of system \eqref{eq: main lam} with $0<\lambda < 1$ does not belong to  $\partial \Omega$.
\noindent
\\
First suppose that $(x_1(t),x_2(t))$ is an $\omega$--periodic solution of \eqref{eq: main lam} and  $x_1^+=R\geq x_2^+,$ let  $\xi\in [0,\omega]$   such that $x_1^+=x_1(\xi)$ and  from the first equation we obtain
$$
x_1'(\xi)=0=\lambda\left[-\dfrac{\delta_{11}(\xi)x_1(\xi)}{c_{11}(\xi)+x_1(\xi)} +b_1(\xi)f(x_1(\xi-\tau_1(\xi)) )+\dfrac{\delta_{12}(\xi)x_2(\xi)}{c_{12}(\xi)+x_2(\xi)}\right]$$
it follows that
$$
\dfrac{\delta_{11}(\xi)R}{c_{11}(\xi)+R} =b_1(\xi)f(x_1(\xi-\tau_1(\xi)) )+\dfrac{\delta_{12}(\xi)x_2(\xi)}{c_{12}(\xi)+x_2(\xi)}.
$$
Since the function $u\mapsto\frac{\delta u}{c+u}$ is increasing  and bounded, the functions $b_1(\cdot), \delta_{12}(\cdot)$ and $c_{12}(\cdot)$ are positive and bounded, and $ f(u)\leq \frac{1}{e}$ for $u\in \mathbb{R}^+$, we have:
\begin{equation}
\dfrac{\delta_{11}^- R}{c_{11}^+ +R}<\dfrac{\delta_{11}(\xi)R}{c_{11}(\xi)+R} < \frac{b_1^+}{e}+\delta_{12}^+,
\end{equation}
it follows
\begin{equation}\label{8}
\dfrac{\delta_{11}^- R}{c_{11}^+ +R}< \frac{b_1^+}{e}+\delta_{12}^+.
\end{equation}
Then, if $R$ takes arbitrarily large values \eqref{8} become into $\delta_{11}^- \leq \dfrac{b_1^+}{e} +\delta_{12}^+$ which contradicts (H1). Consequently there is a positive number $R_0^*$ such that 
$$x_1(t)<R_0^*,\mbox{ for }t\in\mathbb{R}.$$
In analogous way, replacing (H1) by (H2), we can prove the existence of an upper bound  $\tilde{R_0}$ for $x_2.$
Now we define $R_0= \max \{R_0^* , \tilde{R_0}\}$.

\medskip
\noindent
\textit{A priori lower  bounds.} As before, we will prove that if $0<\varepsilon$ is close enough to zero then the $\omega$--periodic solutions of system \eqref{eq: main lam} does not belong to $\partial \Omega$. We consider $\varepsilon=\min\{x_1^-,x_2^-\}$ and suppose that $x_1(\eta)=\varepsilon$ for some $\eta\in [0,\omega]$, then
$$
x_1'(\eta)=0=\lambda\left[-\dfrac{\delta_{11}(\eta)x_1(\eta)}{c_{11}(\eta)+x_1(\eta)} +b_1(\eta)f(x_1(\eta-\tau_1(\eta)) )+\dfrac{\delta_{12}(\eta)x_2(\eta)}{c_{12}(\eta)+x_2(\eta)}\right],$$
hence
$$\dfrac{\delta_{11}(\eta)x_1(\eta)}{c_{11}(\eta)+x_1(\eta)} =b_1(\eta)f(x_1(\eta-\tau_1(\eta)) )+\dfrac{\delta_{12}(\eta)x_2(\eta)}{c_{12}(\eta)+x_2(\eta)}.
$$
We assume that $R_0\geq 1$ and consider $R_1$ as the unique value in $(0,1]$ such that $f(R_1)=f(R_0).$ Suppose that $R_1 <\varepsilon$, so $
R_1 < x_1(t),$ for $t\in\mathbb{R}$ and we have obtained trivially a lower bounds for the solutions of \eqref{eq: main lam}.

On the other hand, we consider that $\varepsilon\leq R_1$, so $\varepsilon \leq x_1(\eta-\tau_1(\eta))\leq R_1$ and 
$$f(x_1(\eta-\tau_1(\eta)))\geq f(\varepsilon),$$
since $f$ is an increasing function on $[0,1)$ and  decreasing on $(1,+\infty)$. Thus
$$\dfrac{\delta_{11}(\eta)\varepsilon}{c_{11}(\eta)+\varepsilon} \geq b_1(\eta)f(\varepsilon )+\dfrac{\delta_{12}(\eta)x_2(\eta)}{c_{12}(\eta)+x_2(\eta)}$$
it follows that
$$\dfrac{\delta_{11}(\eta)\varepsilon}{c_{11}(\eta)+\varepsilon} > b_1(\eta)f(\varepsilon )+\dfrac{\delta_{12}^-x_2(\eta)}{c_{12}^++x_2(\eta)}.$$
Since the function $u\mapsto \frac{\delta_{12}^-u}{c_{12}^++u}$ is increasing in  $u$ on the interval $(-c_{12}^+,\infty)$ and by definition of $f(\varepsilon)$ we obtain
$$\dfrac{\delta_{11}^+\varepsilon}{c_{11}^-+\varepsilon}-\dfrac{\delta_{12}^-\varepsilon}{c_{12}^++\varepsilon} > b_1^-\varepsilon e^{-\varepsilon}$$
equivalent to
\begin{equation}\label{9}
\dfrac{\delta_{11}^+}{c_{11}^-+\varepsilon}-\dfrac{\delta_{12}^-}{c_{12}^++\varepsilon} > b_1^- e^{-\varepsilon}.    
\end{equation}
Then, making tend $\varepsilon$ to zero \eqref{9} become into
$$\dfrac{\delta_{11}^+}{c_{11}^-}-\dfrac{\delta_{12}^-}{c_{12}^+} \geq  b_1^-,$$
which contradicts (H3). Consequently there is $\varepsilon_0 \in (0,R_1)$ such that 
$$\varepsilon_0 < x_1(t),\mbox{ for }t\in\mathbb{R}.$$
With analogous arguments, replacing (H3) by  (H4), we can obtain a lower bound  $\tilde{\epsilon}_0$ for $x_2.$
Note that we have actually proved that there are some open and bounded set $\Omega$ such that condition a) of Lemma 1 holds.\\\\
\noindent\textbf{\textit{b) The function $g$ has no zeroes in $\partial \Omega\cap \mathbb{R}^2$.}}
From \eqref{def Omega} it follows that $\Omega\cap \mathbb{R}^2=]\epsilon,R[\times ]\epsilon,R[$.  We shall prove that there are positive constants $\varepsilon$ and $R$ such that $g(\mathbf{x})\neq 0$  for $\mathbf{x}\in\partial \Omega\cap \mathbb{R}^2$ .
 
We note that
\begin{eqnarray*}
g_1(x_1,x_2)&:=&\dfrac{1}{\omega}\int_0^{\omega}\left(\dfrac{\delta_{11}(t)x_1}{c_{11}(t)+x_1} -b_1(t)f(x_1)-\dfrac{\delta_{12}(t)x_2}{c_{12}(t)+x_2}\right)dt,\\ \\
g_2(x_1,x_2)&:=&\dfrac{1}{\omega}\int_0^{\omega}\left(\dfrac{\delta_{22}(t)x_2}{c_{22}(t)+x_2} -b_2(t)f(x_2 )-\dfrac{\delta_{21}(t)x_1}{c_{21}(t)+x_1}\right)dt.
\end{eqnarray*}
Since the definition of $g_1(x_1,x_2)$  it follows that for $x_1=\varepsilon$ and $\varepsilon\leq x_2\leq R$
\begin{eqnarray*}
g_1(\varepsilon,x_2)&=&\dfrac{\varepsilon}{\omega}\int_0^{\omega}\left(\dfrac{\delta_{11}(t)}{c_{11}(t)+\varepsilon} -b_1(t)e^{-\varepsilon}-\dfrac{\delta_{12}(t)x_2}{c_{12}(t)+x_2}\right)dt\\
&\leq &\dfrac{\varepsilon}{\omega}\int_0^{\omega}\left(\dfrac{\delta_{11}(t)}{c_{11}(t)+\varepsilon} -b_1(t)e^{-\varepsilon}-\dfrac{\delta_{12}(t)}{c_{12}(t)+\varepsilon}\right)dt=g_1(\varepsilon,\varepsilon).
\end{eqnarray*}
In particular we have
$$
 g_1(\varepsilon,x_2)\leq g_1(\varepsilon,\varepsilon)< \varepsilon \left(\dfrac{\delta_{11}^+}{c_{11}^-} -b_1^-e^{-\varepsilon}-\dfrac{\delta_{12}^-}{c_{12}^+ +\varepsilon}\right),   
$$
where $\varepsilon\leq x_2\leq R$. It follows, by assumption (H3) , that there is $\varepsilon_1>0$ such that if $\varepsilon<\varepsilon_1$  the above equation become into 
\begin{equation}\label{eq: 12}
 g_1(\varepsilon,x_2)\leq g_1(\varepsilon,\varepsilon)< 0\mbox{ for }\varepsilon\leq x_2\leq R.
\end{equation}
On the other hand, if $x_1=R$ and $\varepsilon\leq x_2\leq R$ then
\begin{eqnarray*}
g_1(R,x_2)&=&\dfrac{R}{\omega}\int_0^{\omega}\left(\dfrac{\delta_{11}(t)}{c_{11}(t)+R} -b_1(t)e^{-R}-\dfrac{\delta_{12}(t)x_2}{c_{12}(t)+x_2}\right)dt\\
&\geq & \dfrac{R}{\omega}\int_0^{\omega}\left(\dfrac{\delta_{11}(t)}{c_{11}(t)+R} -b_1(t)e^{-R}-\dfrac{\delta_{12}(t)}{c_{12}(t)+R}\right)dt=g_1(R,R).
\end{eqnarray*}
Hence for $\varepsilon\leq x_2\leq R$ we obtain
$$g_1(R,x_2)\geq g_1(R,R)> R \left(\delta_{11}^- -\dfrac{b_1^+}{e}-\delta_{12}^+\right).$$
Assumption (H1) implies that there is $R_1>0$  such that if $R> R_1$ then  
\begin{equation}\label{eq: 13}
 g_1(R,x_2)\geq g_1(R,R)> 0\mbox{ for }\varepsilon\leq x_2\leq R.
\end{equation}
We conclude that if $\varepsilon <\varepsilon_1$ and $R>R_1$,  then
\begin{equation}\label{14}
g_1(\varepsilon, x_2)<0< g_1(R, x_2),\quad \mbox{ for all }\varepsilon\leq x_2 \leq R.    
\end{equation}
Analogously, using definition of $g_2(x_1,x_2)$ and assumptions  (H2) and (H4),  we obtain, 
\begin{equation}\label{15}
g_2(x_1,\varepsilon)<0< g_2(x_1,R),\quad \mbox{ for all }\varepsilon\leq x_1 \leq R,    
\end{equation}
where  $\varepsilon <\varepsilon_1$ and $R>R_1$. We have proved that  $g(\mathbf{x})\neq 0$ for $\mathbf{x}\in \partial \Omega\cap \mathbb{R}^2$.
 
\bigskip
\noindent\textbf{\textit{c) The degree of $g$ at $0$ over  $\Omega\cap \mathbb{R}^2$ is non zero.}}
Since the equations \eqref{14} holds \eqref{15}, combined with the continuity  of $g: \mathbb{R}^2 \to \mathbb{R}^2$, we shall apply Lemma \ref{Poinc Mir},  to conclude that there is  $\mathbf{\tilde{x}}\in \Omega\cap \mathbb{R}^2$ such that $g(\mathbf{\tilde{x}})=0$, so we conclude that $\deg(g,\Omega\cap\mathbb{R}^2,0)\neq 0$.

Then, by lemma \ref{thm EDR resonante }, the system \eqref{eq: main} has at least one solution $X\in \overline{\Omega}.$
\end{proof}

\begin{remark}
Chen and Wang in \cite{chen2012positive} studied the system given by 
\begin{center}
$ \begin{array}{ccc}
x_1'(t)&=&-\dfrac{\delta_{11}(t)x_1(t)}{c_{11}(t)+x_1(t)}+\dfrac{\delta_{12}(t)x_2(t)}{c_{12}(t)+x_2(t)} +\sum_{j=1}^nb_{1j}(t)f(x_1(t-\tau_{1j}(t)) )\\  \\
x_2'(t)&=&-\dfrac{\delta_{22}(t)x_2(t)}{c_{22}(t)+x_2(t)}+\dfrac{\delta_{21}(t)x_1(t)}{c_{21}(t)+x_1(t)} +\sum_{j=1}^n b_{2j}(t)f(x_2(t-\tau_{2j}(t)) ),
\end{array}$
\end{center}
with all the coefficients and delay functions $\omega$--periodic, continuous and positive. Then, as has been developed in our work, it follows that to demonstrate the existence of at least one $\omega $--periodic solution of the previous system, it is sufficient to consider the hypotheses
\begin{multicols}{2}
 \begin{itemize}
      \item[(H1')]\quad $\dfrac{\sum_{j=1}^n b_{1j}^+}{e} +\delta_{12}^+< \delta_{11}^-$,\\
     \item[(H3')] \quad $\dfrac{\delta_{11}^+}{c_{11}^-}-\dfrac{\delta_{12}^-}{c_{12}^+} < \sum_{j=1}^n b_{1j}^-$, \\
     \item[(H2')] \quad $\dfrac{\sum_{j=1}^n b_{2j}^+}{e} +\delta_{21}^+< \delta_{22}^-$,\\
      \item[(H4')]\quad $\dfrac{\delta_{22}^+}{c_{22}^-}-\dfrac{\delta_{21}^-}{c_{21}^+} < \sum_{j=1}^n b_{2j}^-$.
          \end{itemize}
        \end{multicols}

\end{remark}

\section{ Examples}
To illustrate our result, we consider an example of the system \eqref{eq: main},
we will obtain conditions in order to verifies the hypotheses of our Theorem \ref{main thm}
and those of the Theorem 2.2 of \cite{chen2012positive} to compare the results.

Consider the following system: 
\begin{eqnarray}\label{eq: example}
x_1'(t)&=&-\frac{ (5+0.25\sin(t))x_1(t)}{4+x_1(t)} +\alpha\left(1+0.5\cos(t)\right)f(x_1(t-7) )+\frac{(2+0.25\cos(t))x_2(t)}{2+x_2(t)}\nonumber\\  \\
x_2'(t)&=&-\frac{(4+0.5\cos(t))x_2(t)}{4+x_2(t)} +\beta\left(1+0.25\cos(t)\right)f(x_2(t-7) )+\frac{(2+0.5\sin(t))x_1(t)}{3+x_1(t)}\nonumber
\end{eqnarray}
and it is easily see that      
\begin{center}
\begin{tabular}{ccccccccc}
$\delta_{11}^+=5.25$, & & $c_{11}^+=4$, & & $b_1^+=1.5\alpha$, & & $\delta_{12}^+=2.25$, & & $c_{12}^+=2$, \\ \\
$\delta_{11}^-=4.75$, & & $c_{11}^-=4$, & & $b_1^-=0.5\alpha$, & & $\delta_{12}^-=1.75$, & & $c_{12}^-=2$, \\  \\
$\delta_{22}^+=4.5$, & & $c_{22}^+=4$, & & $b_2^+=1.25\beta$, & & $\delta_{21}^+=2.5$, & & $c_{21}^+=3 $, \\  \\
$\delta_{22}^-=3.5$, & & $c_{22}^-=4$, & & $b_2^-=0.75\beta$, & & $\delta_{21}^-=1.5$, & & $c_{21}^-=3 $. 
\end{tabular} 
\end{center}

In order to apply the result of Chen and Wang to the above system,
we shall obtain conditions on the parameters $\alpha$ and $\beta$. We recall the assumptions of Theorem 2.2 of \cite{chen2012positive}
$$\frac{b_1^+}{\delta_{11}^- \, e}+\frac{\delta_{12}^+}{\delta_{11}^-}<1,\quad  \frac{b_2^+}{\delta_{22}^- \, e}+\frac{\delta_{21}^+}{\delta_{22}^-}<1.$$
$$C_i>2D_i,\quad  \ln\left(\frac{2B_i}{A_i}\right)>A_i, \quad i=1,2$$
where
$$A_i=2\int_0^{2\pi}\frac{\delta_{ii}(t)}{c_{ii}(t)}dt, \quad B_i=\int_0^{2\pi}b_i(t)dt, \quad C_i=\int_0^{2\pi}\delta_{ii}(t)dt,$$
for $i=1,2$ and
$$D_1=\int_0^{2\pi}\delta_{12}(t)dt, \quad D_2=\int_0^{2\pi}\delta_{21}(t)dt.$$
We note that for \eqref{eq: example} the above assumption  become into
$$ \frac{6\alpha}{19 e} +\frac{9}{19}<1,\quad  \frac{5\beta}{14 e}+\frac{5}{7}<1.$$
$$C_1=10\pi,\quad  D_1=4\pi,\quad C_2=8\pi,\quad  D_2=4\pi,$$ 
so
$$C_1>2D_1,\quad C_2=2D_2;$$
and
$$ \ln\left(\frac{4\alpha}{5}\right)>5\pi, \quad \ln\left(\beta\right)>4\pi.$$
So the parameters $\alpha$ and $\beta$ must satisfy $1.25e^{5\pi}<\alpha <\frac{5}{3}e$ and $e^{4\pi}<\beta< \frac{4}{5}e$. Therefore the result of \cite{chen2012positive} can not be applied to \eqref{eq: example}.

On the other hand, our assumptions (H1)--(H4) becomes into
$$\frac{1.5\alpha}{e}+2.25<4.75, \quad  \frac{1.25\beta}{e}+2.5<3.5,$$

$$\frac{5.25}{4}-\frac{1.75}{2}<0.5\alpha, \quad  \frac{4.5}{4}-\frac{1.5}{3}<0.75\beta.$$
So the parameters $\alpha$ and $\beta$ must satisfy $\frac{7}{8}<\alpha <\frac{5}{3}e$ and $\frac{5}{6}<\beta< \frac{4}{5}e$. Therefore our result can be applied to \eqref{eq: example}.

The above shows that our hypotheses are less restrictive than those in \cite{chen2012positive}.  Finally,  we note that for $\alpha=2$ and $\beta=3$ the assumptions of our Theorem \ref{main thm} holds, so there is at least one $2\pi$--periodic solution for \eqref{eq: example}. Below we include numerical simulations for this case. 
\begin{figure}[h]\label{figure1}
\centerline{\includegraphics[scale=.4]{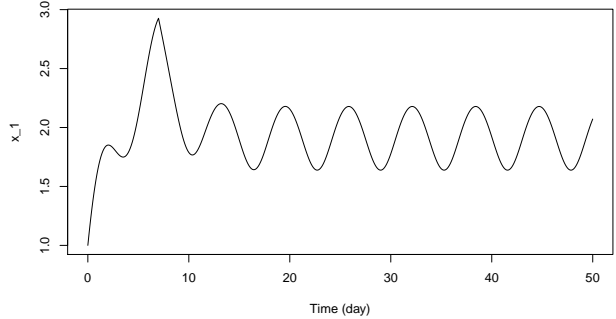}}
\caption{Population Patch 1 of  \eqref{eq: example} with $\alpha=2$ and $\beta=3$.}
\end{figure}

\begin{figure}[h]\label{figure2}
\centerline{\includegraphics[scale=.4]{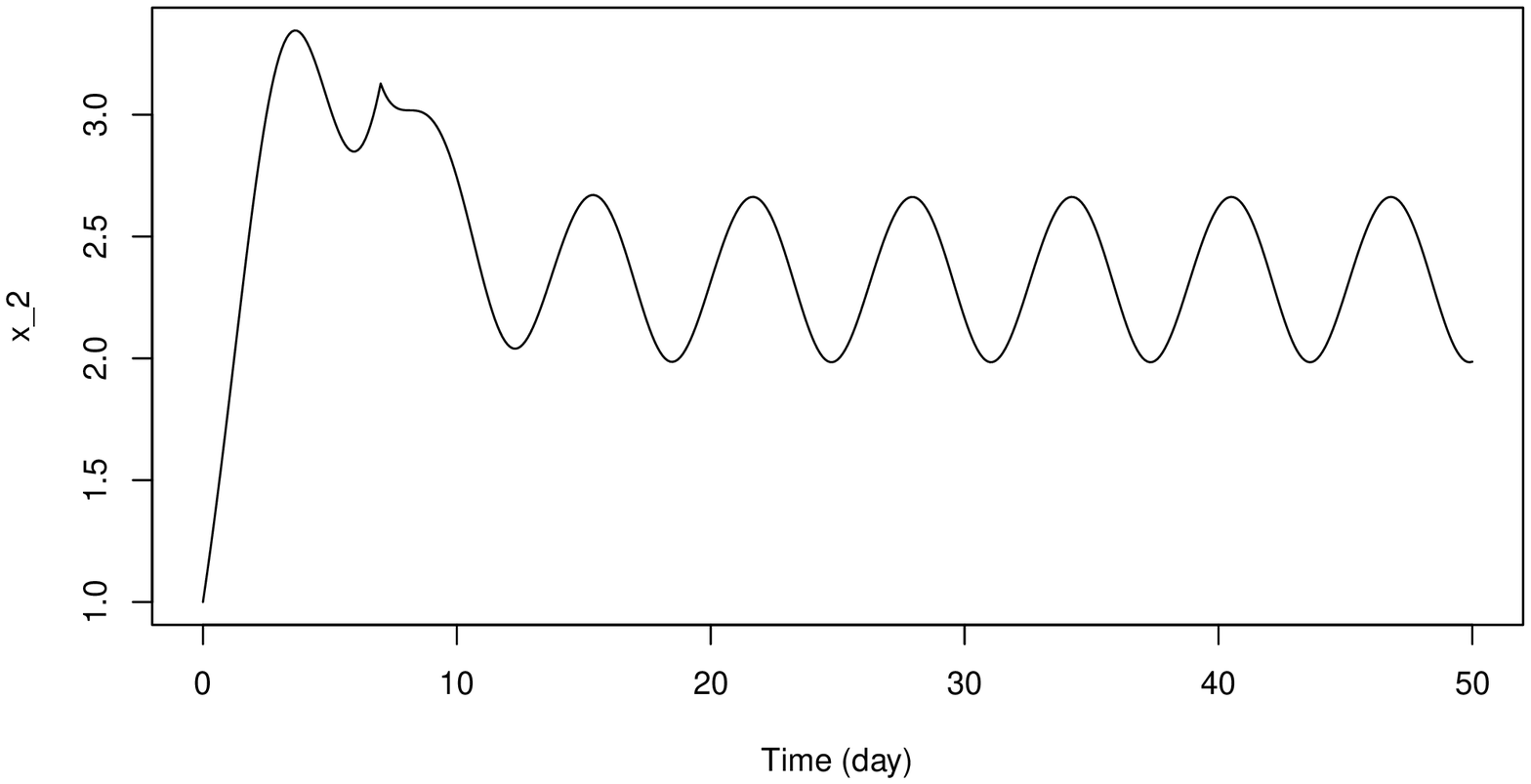}}
\caption{Population Patch 2 of  \eqref{eq: example} with $\alpha=2$ and $\beta=3$.}
\end{figure}

\newpage





\end{document}